\documentclass[12pt]{article}
\usepackage{fullpage}
\usepackage{amsmath,amsthm,amssymb}
\usepackage{boxedminipage}

\newtheorem{assumption}{Assumption}
\newtheorem{definition}{Definition}
\newtheorem{lemma}{Lemma}

\newtheorem{theorem}{Theorem}
\newtheorem{corollary}{Corollary}
\theoremstyle{remark}\newtheorem*{remark}{Remark}

\newcommand{\ba}{\begin{array}}
\newcommand{\ea}{\end{array}}
\newcommand{\beq}{\begin{equation}}
\newcommand{\eeq}{\end{equation}}
\newcommand{\beqa}{\begin{eqnarray}}
\newcommand{\eeqa}{\end{eqnarray}}
\newcommand{\beqas}{\begin{eqnarray*}}
\newcommand{\eeqas}{\end{eqnarray*}}
\newcommand{\bi}{\begin{itemize}}
\newcommand{\ei}{\end{itemize}}
\newcommand{\gap}{\hspace*{2em}}
\newcommand{\reals}{\mathbb{R}}

\newcommand{\dom}{\mathrm{dom\,}}
\newcommand{\eqdef}{\stackrel{\mathrm{def}}{=}}
\newcommand{\bE}{\mathbf{E}}
\newcommand{\bP}{\mathbf{P}}
\newcommand{\diag}{\mathrm{diag}}
\newcommand{\eps}{\epsilon}

\title{On the Complexity Analysis of Randomized Block-Coordinate Descent Methods}

\author{
	Zhaosong Lu%
	\thanks{
	Department of Mathematics, Simon Fraser University, Burnaby, BC,
	V5A 1S6, Canada. (email: {\tt zhaosong@sfu.ca}). This author was supported
        in part by NSERC Discovery Grant.}
	\and
	Lin Xiao
	\thanks{Machine Learning Groups, Microsoft Research, One Microsoft Way, Redmond, WA 98052, USA.
        (email: {\tt lin.xiao@microsoft.com}).}
}

\date{May 20, 2013}

\begin{document}

\maketitle

\begin{abstract}

In this paper we analyze the randomized block-coordinate descent (RBCD) methods proposed
in~\cite{Nesterov12rcdm,RichtarikTakac12} for minimizing the sum of a smooth convex function and a
block-separable convex function. In particular, we extend Nesterov's technique developed in~\cite{Nesterov12rcdm} for analyzing the RBCD method for minimizing a smooth convex
function over a block-separable closed convex set to the aforementioned more general problem and obtain a
sharper expected-value type of convergence rate than
the one implied in~\cite{RichtarikTakac12}.
Also, we obtain a
better high-probability type of iteration complexity, which improves upon the one 
in~\cite{RichtarikTakac12}
by at least the amount $O(n/\epsilon)$, where $\eps$ is the target solution accuracy and~$n$ is
the number of problem blocks.  In addition, for unconstrained smooth convex minimization,  we
develop a new technique called {\it randomized estimate sequence} to analyze the accelerated RBCD
method proposed by Nesterov \cite{Nesterov12rcdm} and establish a sharper expected-value type of
convergence rate than the one given in \cite{Nesterov12rcdm}.

\vskip14pt


\vskip14pt


\end{abstract}

\section{Introduction} \label{introduction}

Block-coordinate descent (BCD) methods and their variants have been successfully applied
to solve various large-scale optimization problems (see, for example, \cite{WuLa08,LiOs09,TsYu09-1,TsYu09-2,WeGoSc,Wright10,QiScGo10,YuTo11}).
 At each iteration, these methods choose one block of coordinates to sufficiently reduce the objective value while keeping the
other blocks fixed. One common and simple approach for choosing such a block is by means of a {\it cyclic} strategy. The global and local convergence of the cyclic BCD method have been well studied in the
literature (see, for example, \cite{Tseng01,LuTs02}) though its global convergence rate still
remains unknown except for some special cases \cite{SahaTewari13}.

Instead of using a deterministic cyclic order, recently many researchers
proposed randomized strategies for choosing a block to update 
at each iteration of the BCD methods
\cite{ChHsLi08,HsChLiKeSu08,ShTe09,LeLe10,Nesterov12rcdm,RiTa11,RichtarikTakac12,ShZh12,
RichtarikTakac12-1,TaRiGo13}. 
The resulting methods are called randomized BCD (RBCD) methods. 
Numerous experiments have demonstrated that the
RBCD methods are very powerful for solving large- and even huge-scale optimization problems
arising in machine learning \cite{ChHsLi08,HsChLiKeSu08,ShTe09,ShZh12}. In particular, Chang et al.\
\cite{ChHsLi08} proposed a RBCD method for minimizing several smooth functions appearing in
machine learning and derived its iteration complexity. Shalev-Shwartz and Tewari \cite{ShTe09} studied
a RBCD method for minimizing $l_1$-regularized smooth convex problems. They first transformed the problem
into a box-constrained smooth problem by doubling the dimension and then applied a block-coordinate
gradient descent method in which each block was chosen with equal probability. Leventhal and Lewis
\cite{LeLe10} proposed a RBCD method for minimizing a convex quadratic function and established its
iteration complexity.  Nesterov \cite{Nesterov12rcdm} analyzed some RBCD methods for minimizing
a smooth convex function over a closed block-separable convex set and established its iteration complexity,
which in effect extends and improves upon some of the results in \cite{ChHsLi08,LeLe10,ShTe09} in
several aspects. 
Richt\'{a}rik and Tak\'{a}\v{c} \cite{RichtarikTakac12} generalized
the RBCD methods proposed in \cite{Nesterov12rcdm} to the problem of minimizing a composite objective 
(i.e., the sum of a smooth convex function and a block-separable convex function) and 
derived some improved complexity results than those given in \cite{Nesterov12rcdm}. 
More recently, Shalev-Shwartz and Zhang \cite{ShZh12} studied a randomized
proximal coordinate ascent method for solving the dual of a class of large-scale convex
minimization problems arising in machine learning and established iteration complexity for obtaining a
pair of approximate primal-dual solutions.

Inspired by the recent work \cite{Nesterov12rcdm,RichtarikTakac12}, we consider the problem of
minimizing the sum of two convex functions:
\begin{equation}\label{eqn:composite-min}
\min_{x\in\Re^N} \ \left\{ F(x) \eqdef f(x) + \Psi(x) \right\},
\end{equation}
where $f$ is differentiable on $\Re^N$, and $\Psi$ has a block separable structure.
More specifically,
\[
\Psi(x) = \sum_{i=1}^n \Psi_i(x_i) ,
\]
where each $x_i$ denotes a subvector of~$x$ with cardinality~$N_i$, 
the collection $\{x_i :i=1,\ldots,n\}$ form a
partition of the components of~$x$, and each~$\Psi_i: \Re^{N_i} \to \Re\cup\{+\infty\}$ is a closed
convex function.
Given the current iterate $x^k$, the RBCD method \cite{RichtarikTakac12}
picks a block~$i\in\{1,\ldots,n\}$ uniformly 
at random and solves a block-wise proximal subproblem in the form of
\[
d_i(x^k) := \arg\min_{d_i\in\Re^{N_i}} \left\{ \langle \nabla_i f(x^k), 
d_i\rangle
+ \frac{L_i}{2}\|d_i\|^2 + \Psi_i(x^k_i + d_i) \right\},
\]
and then it sets the next iterate as
$x^{k+1}_i = x^k_i + d_i(x)$ and $x^{k+1}_j=x^k_j$ for all $j\neq i$.
Here $\nabla_i f(x)$ denotes the \emph{partial gradient} of~$f$ with respect
to~$x_i$, and~$L_i$ is the Lipschitz constant of the partial gradient
(which will be defined precisely later).

Under the assumption that the partial gradients of~$f$ with respect to each
block coordinate are Lipschitz continuous, Nesterov \cite{Nesterov12rcdm}
studied RBCD methods for solving some {\it special}  cases of problem \eqref{eqn:composite-min}.
In particular, for $\Psi \equiv 0$, he proposed a RBCD method in which a random block is chosen per 
iteration according to a uniform or certain non-uniform probability distributions and 
established an expected-value type of convergence rate. In addition, he proposed a RBCD method for
solving \eqref{eqn:composite-min} with each $\Psi_i$ being the indicator function of a closed convex set, 
in which a random block is chosen uniformly at each iteration. He also derived an expected-value type of
convergence rate for this method. It can be observed that the techniques used by Nesterov
to derive these two convergence rates substantially differ from each other, and moreover, for $\Psi \equiv 0$ 
the second rate is much better than the first one. 
(However, the second technique can only work with uniform distribution.)
Recently, Richt\'{a}rik and Tak\'{a}\v{c} \cite{RichtarikTakac12} extended Nesterov's RBCD methods to  
the {\it general} form of problem~\eqref{eqn:composite-min} and established a high-probability type of iteration
complexity. 
Although the expected-value type of convergence rate is not presented explicitly in \cite{RichtarikTakac12}, 
it can be readily obtained from some intermediate result
developed in \cite{RichtarikTakac12} (see Section~\ref{sec:RBCD}  for a detailed discussion). 
Their results can be considered as a generalization of Nesterov's first 
technique mentioned above. 
Given that for $\Psi \equiv 0$ 
Nesterov's second technique can produce a better convergence rate than his first one, 
a natural question is whether his second technique can be extended to 
work with the general setting of problem \eqref{eqn:composite-min} 
and obtain a sharper 
convergence rate than the one implied in \cite{RichtarikTakac12} .

In addition, Nesterov \cite{Nesterov12rcdm}  proposed an accelerated RBCD (ARCD) method for
solving problem \eqref{eqn:composite-min} with $\Psi \equiv 0$ and established an expected-value
type of convergence rate for his method. When $n=1$, this method becomes a deterministic accelerated full gradient method for minimizing smooth convex functions.
When $f$ is a strongly convex function, the convergence rate given in
\cite{Nesterov12rcdm} for $n=1$ is, however, worse than 
the well-known optimal rate shown in \cite[Theorem~2.2.2]{Nesterov04book}.
Then the question is whether a sharper convergence rate for the ARCD method
than the one given in \cite{Nesterov12rcdm} can be established
(which would match the optimal rate for $n=1$).

In this paper, we successfully address the above two questions by
obtaining some sharper convergence rates for the RBCD method for solving problem
\eqref{eqn:composite-min} and for the ARCD method in the case $\Psi\equiv 0$.
First, we extend Nesterov's second technique \cite{Nesterov12rcdm} 
developed for a special case of~\eqref{eqn:composite-min}
to analyze the RBCD method in the general setting,
and obtain a sharper expected-value type of convergence rate 
than the one implied in \cite{RichtarikTakac12}.
We also obtain a better high-probability type of iteration complexity, which
improves upon the one in \cite{RichtarikTakac12} 
at least by the amount $O(n/\epsilon)$, where $\eps$ is the target solution accuracy. 

For unconstrained smooth convex minimization (i.e., $\Psi \equiv 0$),  
we develop a new technique called {\it randomized estimate sequence} to
analyze Nesterov's ARCD method and establish a sharper expected-value type of convergence rate
than the one given in \cite{Nesterov12rcdm}. Especially, for $n=1$, our rate becomes the same 
as the well-known optimal rate achieved by accelerated full gradient method
\cite[Section~2.2]{Nesterov04book}.

This paper is organized as follows. 
In Section~\ref{sec:tech}, we develop some technical results that are
used to analyze the RBCD methods. 
In Section~\ref{sec:RBCD}, we analyze the RBCD method for problem~\eqref{eqn:composite-min} by extending Nesterov's second technique 
\cite{Nesterov12rcdm},
and establish a sharper expected-value type of converge rate as well as improved
high-probability iteration complexity.
In Section~\ref{sec:accl-rcd}, we develop the
randomized estimate sequence technique and use it to derive a sharper expected-value type of converge 
rate for the ARCD method for solving unconstrained smooth convex minimization.

\section{Technical preliminaries}
\label{sec:tech}

In this section we develop some technical results that will be used to analyze the RBCD and ARCD methods subsequently.
Throughout this paper we assume that problem~\eqref{eqn:composite-min} has
a minimum ($F^\star>-\infty$) and
its set of optimal solutions, denoted by $X^*$, is nonempty.

For any partition of $x\in\Re^N$ into $\{x_i\in\Re^{N_i}:i=1,\ldots,n\}$,
there is an $N\times N$ permutation matrix $U$
partitioned as $U=[U_1 \cdots U_n]$, where $U_i\in\Re^{N\times N_i}$, such that
\[
x = \sum_{i=1}^n U_i x_i, \qquad
\mbox{and}\quad x_i = U_i^T x, \quad i=1,\ldots,n.
\]
For any $x\in\Re^N$, the \emph{partial gradient} of~$f$ with respect to~$x_i$
is defined as
\[
\nabla_i f(x) = U_i^T \nabla f(x), \quad i=1,\ldots,n.
\]
For simplicity of presentation, we associate each subspace $\Re^{N_i}$,
for $i=1,\ldots,n$, with the standard Euclidean norm, denoted by $\|\cdot\|$.
We make the following assumption which is used 
in \cite{Nesterov12rcdm,RichtarikTakac12} as well.

\begin{assumption}\label{asmp:coord-smooth}
The gradient of function~$f$ is block-wise Lipschitz continuous
with constants $L_i$, i.e.,
\[
\|\nabla_i f(x+U_i h_i) - \nabla_i f(x) \| \leq L_i\|h_i\|,
\quad \forall\, h_i\in\reals^{N_i}, \quad i=1,\ldots,n, \quad x\in\reals^N.
\]
\end{assumption}

\gap 

Following \cite{Nesterov12rcdm}, we define the following pair of norms in 
the whole space $\Re^N$:
\begin{eqnarray*}
\|x\|_L &=& \biggl(\sum_{i=1}^n L_i \|x_i\|^2\biggr)^{1/2},
\quad\forall\,x\in\Re^N, \\
\|g\|_L^* &=& \biggl(\sum_{i=1}^n \frac{1}{L_i} \|g_i\|^2\biggr)^{1/2},
\quad\forall\,g\in\Re^N.
\end{eqnarray*}
Clearly, they satisfy the Cauchy-Schwartz inequality:
\[
\langle g, x \rangle \leq \|x\|_L \cdot \|g\|_L^*, \quad \forall\,x,g\in\Re^N.
\]

The convexity parameter of a convex function $\phi:\Re^N\to\Re\cup\{+\infty\}$ with
respect to the norm $\|\cdot\|_L$, denoted by $\mu_\phi$, is the largest $\mu \ge 0$
such that for all $x, y\in\dom\phi$,
\[
\phi(y) \geq \phi(x) + \langle s, y-x\rangle + \frac{\mu}{2}
\|y-x\|_L^2, \quad \forall\, s\in\partial\phi(x).
\]
Clearly, $\phi$ is strongly convex if and only if $\mu_\phi>0$.

Assume that  $f$ and~$\Psi$ have convexity parameters $\mu_f \ge 0$ and $\mu_\Psi \ge 0$
with respect to the norm $\|\cdot\|_L$, respectively. Then the convexity parameter of $F=f+\Psi$
is  at least $\mu_f + \mu_\Psi$. Moreover, by Assumption~\ref{asmp:coord-smooth}, we have
\begin{equation} \label{eqn:lip-ineq}
f(x+U_i h_i) \leq f(x) + \langle \nabla_i f(x), h_i\rangle
+ \frac{L_i}{2}\|h_i\|^2,
\quad \forall\, h_i\in\reals^{N_i}, \quad i=1,\ldots,n, \quad x\in\reals^N,
\end{equation}
which immediately implies that $\mu_f\leq 1$.

\gap 

The following lemma concerns the expected value of a block-separable
function when a random block of coordinate is updated.
\begin{lemma}\label{lem:norm-inc}
Suppose that $\Phi(x) = \sum^n_{i=1} \Phi_i(x_i)$. For any $x,d\in\Re^N$, if we pick $i\in\{1,\ldots,n\}$ uniformly
at random, then
\[
\bE_i\bigl[\Phi(x+U_i d_i)\bigr]
=\frac{1}{n}\Phi(x+d) + \frac{n-1}{n} \Phi(x) .
\]
\end{lemma}

\begin{proof}
Since each~$i$ is picked randomly with probability $1/n$, we have
\begin{eqnarray*}
\bE_i\bigl[\Phi(x+U_i d_i)\bigr]
&=& \frac{1}{n} \sum_{i=1}^n \biggl( \Phi_i(x_i+d_i)
    +\sum_{j\neq i} \Phi_j(x_j) \biggr)\\
&=& \frac{1}{n} \sum_{i=1}^n \Phi_i(x_i+d_i)
    +\frac{1}{n}\sum_{i=1}^n \sum_{j\neq i} \Phi_j(x_j) \\
&=& \frac{1}{n}\Phi(x+d) + \frac{n-1}{n} \Phi(x) .
\end{eqnarray*}
\end{proof}

\gap

For notational convenience, we define
\begin{equation}\label{eqn:H-def}
H(x,d) \ := \  f(x) + \langle \nabla f(x), d \rangle + \frac{1}{2}\|d\|_L^2
+ \Psi(x + d).
\end{equation}
The following result is equivalent to \cite[Lemma~2]{RichtarikTakac12}.

\begin{lemma} \label{lem:expected-descent}
Suppose $x,d\in\Re^N$. If we pick $i\in\{1,\ldots,n\}$ uniformly
at random, then
\[
\bE_i \bigl[ F(x+U_id_i) \bigr] - F(x)
~\leq~ \frac{1}{n} \bigl( H(x,d) - F(x) \bigr).
\]
\end{lemma}

\gap

We next develop some results regarding the
\emph{block-wise composite gradient mapping}. 
Composite gradient mapping was introduced by Nesterov \cite{Nesterov07composite}
for the analysis of full gradient methods for solving 
problem~\eqref{eqn:composite-min}.
Here we extend the concept and several associated properties 
to the block-coordinate case.

As mentioned in the introduction, the RBCD methods studied 
in~\cite{RichtarikTakac12} solves in each iteration a block-wise 
proximal subproblem in the form of: 
\[
d_i(x) := \arg\min_{d_i\in\Re^{N_i}} \left\{ \langle \nabla_i f(x), d_i\rangle
+ \frac{L_i}{2}\|d_i\|^2 + \Psi_i(x_i + d_i) \right\}, 
\]
for some $i\in\{1,\ldots,n\}$.
By the first-order optimality condition, there exists a subgradient $s_i\in\partial\Psi_i(x_i+d_i(x))$ such that
\begin{equation}\label{eqn:coord-opt}
\nabla_i f(x) + L_i d_i(x) + s_i = 0.
\end{equation}
Let $d(x) = \sum_{i=1}^n U_i d_i(x)$.
By \eqref{eqn:H-def}, the definition of $\|\cdot\|_L$ and separability of~$\Psi$, we then have
\[
d(x) = \arg\min_{d\in\Re^N} H(x,d).
\]

We define the block-wise composite gradient mappings as
\[
g_i(x) \eqdef - L_i d_i(x), \quad i=1,\ldots,n.
\]
From the optimality conditions~(\ref{eqn:coord-opt}), we conclude
\[
-\nabla_i f(x) + g_i(x) \in \partial \Psi_i(x_i+d_i(x)), \quad i=1,\ldots, n.
\]
Let
\[
g(x) = \sum_{i=1}^n U_i g_i(x).
\]
Then we have
\begin{equation}\label{eqn:coord-subg}
-\nabla f(x) + g(x) \in \partial \Psi(x+d(x)).
\end{equation}
Moreover,
\[
\|d(x)\|_L^2 = \sum_{i=1}^n L_i \|d_i(x)\|^2
=\sum_{i=1}^n \frac{1}{L_i} \|g_i(x)\|^2 = \bigl( \|g(x)\|_L^* \bigr)^2,
\]
and
\begin{equation}\label{eqn:g-d-inner}
\langle g(x), d(x) \rangle = - \|d(x)\|_L^2 = -\bigl( \|g(x)\|_L^* \bigr)^2 .
\end{equation}

The following result establishes a lower bound of the function value $F(y)$,
where $y$ is arbitrary in $\Re^N$, based on the composite gradient mapping
at another point~$x$.

\begin{lemma}\label{lem:forward-looking}
For any fixed $x,y\in\Re^N$, if we pick $i\in\{1,\ldots,n\}$ uniformly
at random, then
\begin{eqnarray*}
\frac{1}{n} F(y) + \frac{n-1}{n} F(x)
&\geq& \bE_i \bigl[F(x+U_i d_i(x))\bigr]
  +\frac{1}{n}\left( \langle g(x), y-x \rangle
  + \frac{1}{2}\bigl(\|g(x)\|_L^* \bigr)^2 \right) \\
& & + \frac{1}{n}\left( \frac{\mu_f}{2}\|x-y\|_L^2
  + \frac{\mu_\Psi}{2}\|x+d(x)-y\|_L^2 \right).
\end{eqnarray*}
\end{lemma}
\begin{proof}
By (\ref{eqn:coord-subg}) and convexity of $f$ and $\Psi$, we have
\begin{eqnarray*}
H(x, d(x))
&=& f(x) + \langle \nabla f(x), d(x) \rangle + \frac{1}{2}\|d(x)\|_L^2
    + \Psi(x + d(x)) \\
&\leq& f(y) + \langle \nabla f(x), x-y\rangle - \frac{\mu_f}{2}\|x-y\|_L^2
    +\langle \nabla f(x), d(x)\rangle + \frac{1}{2}\|d(x)\|_L^2 \\
&&  + \Psi(y) + \langle -\nabla f(x) + g(x), x+d(x)-y\rangle
    - \frac{\mu_\Psi}{2}\|x+d(x)-y\|_L^2 \\
&=& F(y) + \langle g(x), x-y\rangle + \langle g(x), d(x) \rangle
    + \frac{1}{2}\|d(x)\|_L^2  - \frac{\mu_f}{2}\|x-y\|_L^2 \\
& & - \frac{\mu_\Psi}{2}\|x+d(x)-y\|_L^2 \\
&=& F(y) + \langle g(x), x-y\rangle - \frac{1}{2}\bigl(\|g(x)\|_L^*\bigr)^2
    - \frac{\mu_f}{2}\|x-y\|_L^2
    - \frac{\mu_\Psi}{2}\|x+d(x)-y\|_L^2 ,
\end{eqnarray*}
where the last inequality holds due to~(\ref{eqn:g-d-inner}).
This together with Lemma~\ref{lem:expected-descent} yields the desired result.
\end{proof}

\gap

Using Lemma~\ref{lem:norm-inc} with $\Phi(\cdot)=\|\cdot\|_L^2$,
we can rewrite the conclusion of Lemma~\ref{lem:forward-looking} in an 
equivalent form:
\begin{eqnarray}
\frac{1}{n} F(y) + \frac{n-1}{n} F(x) \!\!&+&\!\! \frac{\mu_\Psi}{2}\|x-y\|_L^2
~\geq~ \bE_i \left[F(x+U_i d_i(x))
  + \frac{\mu_\Psi}{2}\|x+U_id_i-y\|_L^2\right] \nonumber \\
&+&\frac{1}{n}\left( \langle g(x), y-x \rangle
  + \frac{1}{2}\bigl(\|g(x)\|_L^* \bigr)^2
  + \frac{\mu_f+\mu_\Psi}{2}\|x-y\|_L^2 \right).
\label{eqn:forward-messy}
\end{eqnarray}
This is the form we will actually use in our subsequent convergence analysis.

\gap

Letting $y=x$ in Lemma~\ref{lem:forward-looking}, we obtain the following
corollary.

\begin{corollary} \label{cor:monotone}
Given $x\in\Re^N$. If we pick $i\in\{1,\ldots,n\}$ uniformly
at random, then
\[
F(x) - \bE_i\bigl[ F(x+U_i d_i(x))\bigr]
~\geq~  \frac{1+\mu_\Psi}{2n} \bigl(\|g(x)\|_L^*)^2
~=~  \frac{1+\mu_\Psi}{2n} \bigl(\|d(x)\|_L)^2.
\]
\end{corollary}

\gap

By similar arguments as in the proof of Lemma~\ref{lem:forward-looking}, it can be shown that a 
similar result as Lemma~\ref{lem:forward-looking} also holds block-wise without taking expectation:
\[
F(x) - F(x+U_i d_i(x)) \geq \frac{1+\mu_\Psi}{2} L_i \|d_i(x)\|^2 .
\]

The following (trivial) corollary is useful when we do not have knowledge
on~$\mu_f$ or~$\mu_\Psi$.
\begin{corollary}\label{lem:forward-looking-nomu}
For any fixed $x,y\in\Re^N$, if we pick $i\in\{1,\ldots,n\}$ uniformly
at random, then
\[
\frac{1}{n} F(y) + \frac{n-1}{n} F(x)
~\geq~ \bE_i \bigl[F(x+U_i d_i(x))\bigr]
  +\frac{1}{n}\left( \langle g(x), y-x \rangle
  + \frac{1}{2}\bigl(\|g(x)\|_L^* \bigr)^2 \right).
\]
\end{corollary}

\section{Randomized block-coordinate descent}
\label{sec:RBCD}

In this section we analyze the following randomized block coordinate descent (RBCD) method
for solving problem \eqref{eqn:composite-min}, which was proposed in \cite{RichtarikTakac12}. 
In particular, we extend Nesterov's technique \cite{Nesterov12rcdm} developed for a special case 
of problem \eqref{eqn:composite-min} to work with the general setting and establish some sharper 
expected-value type of converge rate, as well as improved high-probability iteration complexity, than those 
given or implied in \cite{RichtarikTakac12}.

\begin{center}
\begin{boxedminipage}{0.75\textwidth}
\vspace{1ex}
\centerline{\textbf{Algorithm:} RBCD$(x^0)$}
\vspace{1ex}
Repeat for $k=0,1,2,\ldots$
\vspace{-1ex}
\begin{enumerate} \itemsep 0pt
\item Choose $i_k\in\{1,\ldots,n\}$ randomly with a uniform distribution.
\item Update $x^{k+1} = x^{k} + U_{i_k} d_{i_k}(x^{k})$.
\end{enumerate}
\end{boxedminipage}
\end{center}

After~$k$ iterations, the RBCD method generates a random output
$x^k$, which depends on the observed realization of the random
variable
\[
\xi_{k-1} \eqdef \{i_0, i_1, \ldots, i_{k-1}\}.
\]
The following quantity measures the distance between $x^0$ and the optimal solution
set of problem~(\ref{eqn:composite-min}) that will appear in our complexity
results:
\beq \label{R0}
R_0 \eqdef \min\limits_{x^\star\in X^*} \|x^0-x^\star\|_L,
\eeq
where $X^*$ is the set of optimal solutions of problem~(\ref{eqn:composite-min}).

\subsection{Convergence rate of expected values}
The following theorem is a generalization of \cite[Theorem~5]{Nesterov12rcdm},
where the function~$\Psi$ in~\eqref{eqn:composite-min} is restricted to be
the indicator function of a block-separable closed convex set.
Here we extend it to the general case of $\Psi$ being block-separable 
convex functions by employing the machinery of block-wise composite gradient mapping
developed in Section~\ref{sec:tech}.

\begin{theorem} \label{thm:RBCD-rate}
Let $R_0$ be defined in \eqref{R0}, $F^\star$ be the optimal value of
problem~(\ref{eqn:composite-min}),  and $\{x^k\}$ be the sequence generated by the
RBCD method. Then for any $k\geq 0$, the iterate $x^k$ satisfies
\beq \label{general-complexity}
\bE_{\xi_{k-1}} \bigl[ F(x^k) \bigr] - F^\star ~\leq~ \frac{n}{n+k}
\left(\frac{1}{2}R_0^2 + F(x^0)-F^\star \right).
\eeq
Furthermore, if at least one of $f$ and $\Psi$ is strongly convex, i.e.,
$\mu_f+\mu_\Psi>0$, then
\beq \label{strong-complexity}
\bE_{\xi_{k-1}} \bigl[ F(x^k) \bigr] - F^\star ~\leq~
\left(1-\frac{2(\mu_f+\mu_\Psi)}{n(1+\mu_f+2\mu_\Psi)}\right)^k
\left(\frac{1+\mu_\Psi}{2}R_0^2 + F(x^0)-F^\star \right).
\eeq
\end{theorem}

\begin{proof}
Let $x^\star$ be an arbitrary optimal solution of~(\ref{eqn:composite-min}).
Denote
\[
r_k^2 ~=~ \| x^k - x^\star\|_L^2
~=~ \sum_{i=1}^n L_i \langle x^k_i-x^\star_i, x^k_i-x^\star_i \rangle.
\]
Notice that $x^{k+1} = x^{k} + U_{i_k} d_{i_k}(x^{k})$. Thus we have
\[
r_{k+1}^2 ~=~ r_k^2 + 2 L_{i_k} \langle d_{i_k}(x^k), x^k_{i_k} -x^\star_{i_k}
\rangle + L_{i_k} \|d_{i_k}(x^k)\|^2 .
\]
Multiplying both sides by $1/2$ and taking expectation with respect to~$i_k$ yield
\begin{eqnarray}
\bE_{i_k} \biggl[ \frac{1}{2} r_{k+1}^2 \biggr]
&=& \frac{1}{2} r_k^2 +  \frac{1}{n} \left( \sum_{i=1}^n
     L_i \langle d_i(x^k), x^k_i -x^\star_i\rangle
    +\frac{1}{2} \sum_{i=1}^n \frac{1}{L_i} \|g_i(x^k)\|^2 \right) \nonumber \\
&=& \frac{1}{2} r_k^2 + \frac{1}{n} \left( \langle g(x^k), x^\star - x^k
    \rangle + \frac{1}{2}\bigl(\|g(x^k)\|_L^*\bigr)^2 \right) .
\label{eqn:expected-rk}
\end{eqnarray}
Using Corollary~\ref{lem:forward-looking-nomu}, we obtain
\[
\bE_{i_k} \biggl[\frac{1}{2} r_{k+1}^2\biggr] ~\leq~ \frac{1}{2} r_k^2
+\frac{1}{n} F^\star + \frac{n-1}{n}F(x^k) - \bE_{i_k} F(x^{k+1}).
\]
By rearranging terms, we obtain that for each $k\geq 0$,
\[
\bE_{i_k} \left[ \frac{1}{2} r_{k+1}^2 + F(x^{k+1})-F^\star\right]
~\leq~ \left( \frac{1}{2} r_k^2 + F(x^k)-F^\star \right)
-\frac{1}{n} \left(F(x^k)-F^\star \right).
\]
Taking expectation with respect to $\xi_{k-1}$ on both sides of the above
inequality, we have
\[
\bE_{\xi_k} \left[ \frac{1}{2}r_{k+1}^2 + F(x^{k+1})-F^\star \right] \ \le
\ \bE_{\xi_{k-1}}\left[ \frac{1}{2} r_k^2 + F(x^k)-F^\star \right]
-\frac{1}{n} \bE_{\xi_{k-1}}\left[F(x^k)-F^\star \right].
\]
Applying this inequality recursively and using the fact that
$\bE_{\xi_k}\bigl[ F(x^j)\bigr]$ is monotonically decreasing
for $j=0,\ldots,k+1$ (see Corollary~\ref{cor:monotone}), we further obtain that
\begin{eqnarray*}
\bE_{\xi_k} \bigl[ F(x^{k+1}) \bigr] -F^\star
&\leq& \bE_{\xi_k} \left[ \frac{1}{2}r_{k+1}^2 + F(x^{k+1})-F^\star \right] \\
&\leq& \frac{1}{2} r_0^2 + F(x^0)-F^\star - \frac{1}{n} \sum_{j=0}^k
       \left(\bE_{\xi_k}\bigl[ F(x^j)\bigr] -F^\star\right) \\
&\leq& \frac{1}{2} r_0^2 + F(x^0)-F^\star - \frac{k+1}{n}
       \left(\bE_{\xi_k}\bigl[ F(x^{k+1})\bigr] -F^\star\right).
\end{eqnarray*}
This leads to
\[
\bE_{\xi_k} \bigl[ F(x^{k+1}) \bigr] - F^\star ~\leq~ \frac{n}{n+k+1}
\left(\frac{1}{2}\|x^0-x^\star\|_L^2 + F(x^0)-F^\star \right),
\]
which together with the arbitrariness of $x^\star$ and the definition of
$R_0$ yields \eqref{general-complexity}.

Next we prove \eqref{strong-complexity} under the strong convexity assumption
$\mu_f+\mu_\Psi>0$. Using \eqref{eqn:forward-messy} and \eqref{eqn:expected-rk},
we obtain that
\begin{eqnarray}
\bE_{i_k} \left[ \frac{1+\mu_\Psi}{2} r_{k+1}^2 + F(x^{k+1}) - F^\star\right]
&\leq& \left(\frac{1+\mu_\Psi}{2} r_k^2 + F(x^k) - F^\star\right) \nonumber\\
&& - \frac{1}{n}\left(\frac{\mu_f+\mu_\Psi}{2} r_k^2 + F(x^k) - F^\star\right).
\label{eqn:expected-rk-F}
\end{eqnarray}
By strong convexity of~$F$, we have
\[
\frac{\mu_f+\mu_\Psi}{2} r_k^2 + F(x^k) - F^\star
~\geq~ \frac{\mu_f+\mu_\Psi}{2} r_k^2 + \frac{\mu_f+\mu_\Psi}{2} r_k^2
~=~ (\mu_f+\mu_\Psi) r_k^2 .
\]
Define
\[
\beta = \frac{2(\mu_f+\mu_\Psi)}{1+\mu_f+2\mu_\Psi}.
\]
We have $0<\beta\leq 1$ due to $\mu_f+\mu_\Psi>0$ and $\mu_f\leq 1$.
Then
\begin{eqnarray*}
\frac{\mu_f+\mu_\Psi}{2} r_k^2 + F(x^k) - F^\star
&\geq& \beta\left(\frac{\mu_f+\mu_\Psi}{2} r_k^2 + F(x^k) - F^\star\right)
       +(1-\beta) (\mu_f+\mu_\Psi) r_k^2 \\
&=& \beta \left(\frac{1+\mu_\Psi}{2} r_k^2 + F(x^k) - F^\star\right).
\end{eqnarray*}
Combining the above inequality with~(\ref{eqn:expected-rk-F}) gives
\[
\bE_{i_k} \left[ \frac{1+\mu_\Psi}{2} r_{k+1}^2 + F(x^{k+1}) - F^\star\right]
~\leq~ \left(1-\frac{\beta}{n}\right)
\left(\frac{1+\mu_\Psi}{2} r_k^2 + F(x^k) - F^\star\right)
\]
Taking expectation with respect $\xi_{k-1}$ on both sides of the above relation, we have
\[
\bE_{\xi_k} \left[ \frac{1+\mu_\Psi}{2} r_{k+1}^2 + F(x^{k+1}) - F^\star\right]
~\leq~ \left(1-\frac{\beta}{n}\right)^{k+1}
\left(\frac{1+\mu_\Psi}{2} r_0^2 + F(x^0) - F^\star\right),
\]
which together with the arbitrariness of $x^\star$ and the definition of
$R_0$ leads to \eqref{strong-complexity}.
\end{proof}

\gap

We have the following remarks on comparing the results in Theorem~\ref{thm:RBCD-rate} with those in \cite{RichtarikTakac12}.
\begin{itemize}
\item 
For the {\it general} setting of problem \eqref{eqn:composite-min}, expected-value type of
convergence rate is not presented explicitly in \cite{RichtarikTakac12}. Nevertheless, it can be
derived straightforwardly from the following relation that was proved in \cite[Theorem~5]{RichtarikTakac12}:
\beq \label{Rich-bdd}
\bE_{i_k}[\Delta_{k+1}] \ \le \ \Delta_k - \frac{\Delta^2_k}{2nc}, \quad\quad \forall k \ge 0,
\eeq
where $\Delta_k := F(x^k)-F^\star$, and
\beqa
c &:=& \max\{\bar R_0^2,~F(x^0)-F^\star\}, \label{eqn:c} \\
\bar R_0 &:=& \max_x \Bigl\{ \max_{x^\star \in X^*} \|x-x^\star\|_L:
 F(x) \le F(x^0) \Bigr\}. 
\label{c-bR0}
\eeqa
Taking expectation with respect to $\xi_{k-1}$ on both sides of \eqref{Rich-bdd}, one can have
\[
\bE_{\xi_k}[\Delta_{k+1}] \ \le \ \bE_{\xi_{k-1}}[\Delta_k] - \frac{1}{2nc} \left(\bE_{\xi_{k-1}}[\Delta_k]\right)^2,
\quad\quad \forall k \ge 0.
\]
By this relation and a similar argument as used in the proof of \cite[Theorem~1]{Nesterov12rcdm}, one can
obtain that
\beq \label{Rich-expbdd}
\bE_{\xi_{k-1}} [F(x^k)]-F^\star \ \le \ \frac{2nc(F(x^0)-F^\star)}{k(F(x^0)-F^\star)+2nc},
\quad\quad \forall k \ge 0.
\eeq
Let $a$ and $b$ denote the right-hand side of \eqref{general-complexity} and \eqref{Rich-expbdd},
respectively. By the definition of $c$ and the relation $\bar R_0 \ge R_0$,  we can see that when
$k$ is sufficiently large,
\[
\frac{b}{a} \ \approx \ \frac{2c}{\frac{1}{2}R_0^2 + F(x^0)-F^\star} \ \ge \ \frac{4}{3}.
\]
Therefore, our expected-value type of convergence rate is better 
by at least a factor of $4/3$ asymptotically,
and the improvement can be much larger if~$\bar R_0$ is much larger
than~$R_0$.

\item 
For the {\it special} case of \eqref{eqn:composite-min} where at least one of $f$ and $\Psi$ is 
strongly convex, i.e., $\mu_f+\mu_\Psi>0$, Richt\'{a}rik and Tak\'{a}\v{c} \cite[Theorem~7]{RichtarikTakac12} 
showed that for all $k \ge 0$, there holds
\[
\bE_{\xi_{k-1}} \bigl[ F(x^k) \bigr] - F^\star ~\leq~
\left(1-\frac{\mu_f+\mu_\Psi}{n(1+\mu_\Psi)}\right)^k
\left(F(x^0)-F^\star \right).
\]
It is not hard to observe that
\beq \label{compare-coef}
\frac{2(\mu_f+\mu_\Psi)}{n(1+\mu_f+2\mu_\Psi)} \ > \ \frac{\mu_f+\mu_\Psi}{n(1+\mu_\Psi)}.
\eeq
It then follows that for sufficiently large $k$, one has
\begin{eqnarray*}
&& \left(1-\frac{2(\mu_f+\mu_\Psi)}{n(1+\mu_f+2\mu_\Psi)}\right)^k
\left(\frac{1+\mu_\Psi}{2}R_0^2 + F(x^0)-F^\star \right) \\
&\leq& \left(1-\frac{2(\mu_f+\mu_\Psi)}{n(1+\mu_f+2\mu_\Psi)}\right)^k
\left(\frac{1+\mu_f+\mu_\Psi}{\mu_f+\mu_\Psi}\right) 
\left(F(x^0)-F^\star \right) \\
& \ll & \left(1-\frac{\mu_f+\mu_\Psi}{n(1+\mu_\Psi)}\right)^k
\left(F(x^0)-F^\star \right).
\end{eqnarray*}
Therefore, our convergence rate \eqref{strong-complexity} is much sharper than their rate
for sufficiently large $k$.
\end{itemize}

\subsection{High probability complexity bound}

By virtue of Theorem \ref{thm:RBCD-rate} we can also derive a sharper iteration complexity
for a {\it single run} of the RBCD method for obtaining an $\epsilon$-optimal solution with
high probability than the one given in \cite[Theorems~5 and 7]{RichtarikTakac12}.

\begin{theorem} \label{prob-complexity}
 Let $R_0$ be defined in \eqref{R0} and $\{x^k\}$ be the sequence generated
 by the RBCD method. Let $0<\epsilon <F(x^0)-F^\star$ and $\rho \in (0,1)$
 be chosen arbitrarily.
 \bi
 \item[(i)] For all $k \ge K$, there holds
 \beq \label{high-prob}
 \bP(F(x^k)-F^\star \le \epsilon) \ \ge \ 1-\rho,
 \eeq
 where
 \beq
 K :=  \frac{2nc}{\epsilon}\left(1+\log\left(\frac{R^2_0+2[F(x^0)-F^\star]}{4c\rho}\right)\right)+2-n.
 \label{general-K}
 \eeq
\item[(ii)] Furthermore, if at least one of $f$ and $\Psi$ is strongly convex, i.e., $\mu_f+\mu_\Psi>0$, then
\eqref{high-prob} holds when $k \ge \tilde K$, where
\[
\tilde K := \frac{n(1+\mu_f+2\mu_\Psi)}{2(\mu_f+\mu_\Psi)}\log\left(\frac{\frac{1+\mu_\Psi}{2}R_0^2 + F(x^0)-F^\star}{\rho\eps}\right)
\]
 \ei
\end{theorem}

\begin{proof}
(i) For convenience, let $\Delta_k = F(x^k)-F^\star$ for all $k$.
Define the truncated sequence $\{\Delta^{\epsilon}_k\}$ as follows:
\[
\Delta^{\epsilon}_k = \left\{
\ba{ll}
\Delta_k & \ \mbox{if} \ \Delta_k \ge \epsilon, \\
0 & \ \mbox{otherwise}.
\ea
\right.
\]
Using \eqref{Rich-bdd} and the same argument as used in the proof of \cite[Theorem~1]{RichtarikTakac12},
one can have
\[
\bE_{i_k}[\Delta^\eps_{k+1}] \ \le \ \left(1-\frac{\eps}{2nc}\right)\Delta^\eps_k, \quad\quad \forall k \ge 0.
\]
Taking expectation with respect to $\xi_{k-1}$ on both sides of the above relation, we obtain that
\beq \label{Rich-bdd1}
\bE_{\xi_k}[\Delta^\eps_{k+1}] \ \le \ \left(1-\frac{\eps}{2nc}\right)\bE_{\xi_{k-1}}[\Delta^\eps_k], \quad\quad \forall k \ge 0.
\eeq
In addition, using \eqref{general-complexity} and the relation $\Delta^\eps_k \le \Delta_k$, we have
\beq \label{We-bdd}
\bE_{\xi_{k-1}}[\Delta^\eps_k]  ~\le~
\frac{n}{n+k} \left(\frac{1}{2}R_0^2 + F(x^0)-F^\star \right), \quad\quad \forall k \ge 0.
\eeq
For any $t>0$, let
\[
K_1 = \left\lceil\frac{n}{t\eps} \left(\frac{1}{2}R_0^2 + F(x^0)-F^\star\right)\right\rceil - n, \quad\quad
K_2 = \left\lceil\frac{2nc}{\eps} \log\left(\frac{t}{\rho}\right)\right\rceil.
\]
It follows from \eqref{We-bdd} that $\bE_{\xi_{K_1-1}}[\Delta^\eps_{K_1}] ~\le~ t\eps$, which together with
\eqref{Rich-bdd1} implies that
\[
\bE_{\xi_{K_1+K_2-1}}[\Delta^\eps_{K_1+K_2}] ~\le~ \left(1-\frac{\eps}{2nc}\right)^{K_2}\bE_{\xi_{K_1-1}}[\Delta^\eps_{K_1}]
~\le~ \left(1-\frac{\eps}{2nc}\right)^{K_2} t \eps~\le~ \rho \eps.
\]
Notice from \eqref{Rich-bdd1} that $\{\bE_{\xi_{k-1}}[\Delta^\eps_k]\}$ is decreasing. Hence, we have
\beq \label{bdd-deltak}
\bE_{\xi_{k-1}}[\Delta^\eps_k] \le \rho \eps, \quad \forall k \ge K(t),
\eeq
where
\[
K(t) := \frac{n}{t\eps} \left(\frac{1}{2}R_0^2 + F(x^0)-F^\star\right) + \frac{2nc}{\eps} \log\left(\frac{t}{\rho}\right)+2-n.
\]
It is not hard to verify that
\[
t^* ~:=~ \frac{\frac12 R^2_0+F(x^0)-F^\star}{2c} ~=~ \arg\min_{t>0} K(t).
\]
Also, one can observe from \eqref{general-K} that $K \ge K(t^*)$, which together with \eqref{bdd-deltak}
implies that
\[
\bE_{\xi_{k-1}}[\Delta^\eps_k] ~\le~ \rho \eps, \quad \forall k \ge K.
\]
Using this relation and Markov inequality, we obtain that
\[
\bP(F(x^k)-F^\star > \epsilon) \ = \ \bP(\Delta_k > \eps) \ = \  \bP(\Delta^\eps_k > \eps)
\ \le \ \frac{\bE_{\xi_{k-1}}[\Delta^\eps_k]}{\eps} \ \le \ \rho, \quad \forall k \ge K,
\]
which immediately implies statement (i) holds.

(ii) 
Using the Markov inequality, the inequality \eqref{strong-complexity} and
the definition of $\tilde K$, we obtain that for any $k \ge \tilde K$,
\beqas
\bP(F(x^k)-F^\star > \epsilon) 
& \le & \frac{\bE_{\xi_{k-1}}[F(x^k)-F^\star]}{\eps} \\
& \le &  \frac{1}{\eps}\left(1-\frac{2(\mu_f+\mu_\Psi)}{n(1+\mu_f+2\mu_\Psi)}\right)^{\tilde K}
\left(\frac{1+\mu_\Psi}{2}R_0^2 + F(x^0)-F^\star \right) \\
& \le &  \frac{1}{\eps} \exp\left(-\frac{2(\mu_f+\mu_\Psi)\tilde K}{n(1+\mu_f+2\mu_\Psi)}\right)
\left(\frac{1+\mu_\Psi}{2}R_0^2 + F(x^0)-F^\star \right) \\
& \le & \rho
\eeqas
and hence statement (ii) holds.
\end{proof}

\gap

We make the following remarks in comparing our results in 
Theorem~\ref{prob-complexity} with those in \cite{RichtarikTakac12}.
\begin{itemize}
\item
For any $0<\eps < F(x^0)-F^\star$ and $\rho \in (0,1)$, Richt\'{a}rik and Tak\'{a}\v{c}
\cite[Theorem~5]{RichtarikTakac12} showed that \eqref{high-prob} holds for all $k \ge \bar K$, where
\[
\bar K =  \frac{2nc}{\epsilon}\left(1+\log\frac{1}{\rho}\right)
+2-\frac{2nc}{F(x^0)-F^\star}
\]
and $c$ is given in \eqref{eqn:c}. 
Using the definitions of $c$ and $R_0$ and the fact
$R_0 \le \bar R_0$, one can observe that
\[
\tau: = \frac{R^2_0+2[F(x^0)-F^\star]}{4c} \ \le \ \frac{3}{4}.
\]
By the definitions of $K$ and $\bar K$, we have that for sufficiently small $\eps>0$,
\[
K - \bar K \ \approx \  \frac{2nc\log \tau}{\epsilon}  \ \le \ -\frac{2nc\log(4/3)}{\epsilon} .
\]
In addition, by the definitions of $R_0$ and $\bar R_0$, one can see that $R_0$ can be much smaller
than $\bar R_0$ and thus $\tau$ can be very small. It follows from the above relation that $K$
can be substantially smaller than $\bar K$. 

\item 
For a {\it special} case of \eqref{eqn:composite-min} where at least one of $f$ and $\Psi$ is strongly convex,
i.e., $\mu_f+\mu_\Psi>0$, Richt\'{a}rik and Tak\'{a}\v{c} \cite[Theorem~8]{RichtarikTakac12} showed that
\eqref{high-prob} holds for all $k \ge \hat K$, where
\[
\hat K :=  \frac{n(1+\mu_\Psi)}{\mu_f+\mu_\Psi}\log\left(\frac{F(x^0)-F^\star}{\rho\eps}\right).
\]
We then see that when $\rho$ or $\eps$ is sufficiently small,
\[
\frac{\tilde K}{\hat K} \ \approx \ \frac{1+\mu_f+2\mu_\Psi}{2(1+\mu_\Psi)}  \ \le \ 1
\]
due to $0\le \mu_f \le 1$. When $\mu_f <1$, we have $\tilde K \le \tilde \tau \hat K$ for some
$\tilde \tau \in (0,1)$ and thus our complexity bound is tighter when $\rho$ or $\eps$ is
sufficiently small.
\end{itemize}

As discussed in \cite[Section~2]{RichtarikTakac12}, the number of iterations required by the RBCD
method for obtaining an $\epsilon$-optimal solution with high probability can also be estimated by
using a  {\it multiple-run} strategy, each run with an independently generated random sequence
$\{i_0,i_1, \ldots\}$. We next derive such an iteration complexity.

\begin{theorem} \label{multi-run-prob}
Let $0<\epsilon <F(x^0)-F^\star$ and $\rho \in (0,1)$ be arbitrarily chosen, and let
$r = \lceil \log(1/\rho)\rceil$. Suppose that we run the RBCD method starting with $x^0$
for $r$ times independently, each time for the same number of iterations $k$. Let $x^k_{(j)}$ denote
the output by the RBCD at the $k$th iteration of the $j$th run. Then there holds:
\[
\bP\left(\min_{1 \le j \le r} F(x^k_{(j)})-F^\star \ \le \ \eps\right) \ \ge \ 1-\rho
\]
for any $k \ge \underline K$, where
\[
\underline K := \left\lceil\frac{en}{\eps} \left(\frac{1}{2}R_0^2 + F(x^0)-F^\star\right)\right\rceil - n.
\]
\end{theorem}

\begin{proof}
Let $\xi^{(j)}_{k-1}=\left\{i^{(j)}_0,i^{(j)}_1,\ldots,i^{(j)}_{k-1}\right\}$ denote the random
sequence used in the $j$th run. Using Markov inequality, \eqref{general-complexity} and the definition of
$\underline K$, we obtain that for any $k \ge \underline K$,
\[
\bP\left(F(x^k_{(j)})-F^\star \ > \ \eps\right) \ \le \ \frac{\bE_{\xi^{(j)}_{k-1}}[F(x^k_{(j)})-F^\star]}{\eps}
\ \le \ \frac{n}{(n+k)\eps}
\left(\frac{1}{2}R_0^2 + F(x^0)-F^\star \right) \ \le \ \frac{1}{e}.
\]
This together with the definition of $r$ implies that
\[
\bP\left(\min_{1 \le j \le r} F(x^k_{(j)})-F^\star \ > \ \eps\right) \ = \
 \Pi^r_{j=1} \bP\left(F(x^k_{(j)})-F^\star \ > \ \eps\right) \ \le \ \frac{1}{e^r} \ \le \ \rho,
\]
and hence the conclusion holds.
\end{proof}

\gap

\begin{remark} 
From Theorem~\ref{multi-run-prob}, one can see that the total number of iterations by
RBCD with a multiple-run strategy for obtaining an $\eps$-optimal solution is at most
\[
K^{\rm M} := \left(\left\lceil\frac{2en}{\eps} \left(R_0^2 + 2(F(x^0)-F^\star)\right)\right\rceil - n
\right)\left\lceil \log\frac{1}{\rho}\right\rceil.
\]
It was implicitly established in \cite{RichtarikTakac12}  that an $\eps$-optimal solution
can be found by RBCD with a multiple-run strategy in at most
\[
\bar K^{\rm M} :=  \left\lceil\frac{2enc}{\epsilon}-\frac{2nc}{F(x^0)-F^\star}\right\rceil
\left\lceil \log\frac{1}{\rho}\right\rceil
\]
iterations. When $\rho$ or $\eps$ is sufficiently small, we have
\[
\frac{K^{\rm M}}{\bar K^{\rm M}} \ \approx \ \frac{R_0^2 + 2(F(x^0)-F^\star)}{c}.
\]
Recall that  $\bar R_0$ can be much larger than $R_0$, which together with \eqref{c-bR0}
implies that $c$ can be much larger than $R_0^2 + 2(F(x^0)-F^\star)$. It follows from the
above relation that when $\rho$ or $\eps$ is sufficiently small, $K^{\rm M}$ can be substantially
smaller than $\bar K^{\rm M}$.
\end{remark}

\section{Accelerated randomized coordinate descent}
\label{sec:accl-rcd}
In this section, we restrict ourselves to the unconstrained
smooth minimization problem
\begin{equation}\label{eqn:unconstr-min}
\min_{x\in\Re^N} \ f(x),
\end{equation}
where $f$ is convex in $\Re^N$ with convexity parameter $\mu = \mu_f \ge 0$  with respect to
the norm $\|\cdot\|_L$ and satisfies Assumption~\ref{asmp:coord-smooth}. It then follows from
 \eqref{eqn:lip-ineq} that $\mu \le 1$. Our aim is to analyze the convergence rate of the
following accelerated randomized coordinate descent (ARCD) method.

\begin{center}
\begin{boxedminipage}{0.8\textwidth}
\vspace{1ex}
\centerline{\textbf{Algorithm:} ARCD$(x^0)$}
\vspace{1ex}
Set $v^0=x^0$, choose $\gamma_0 >0$ arbitrarily,
and repeat for $k=0,1,2,\ldots$
\vspace{-1ex}
\begin{enumerate} \itemsep 0pt
\item Compute $\alpha_k\in(0, n]$ from the equation
\[
\textstyle
\alpha_k^2 = \left(1-\frac{\alpha_k}{n}\right) \gamma_k+\frac{\alpha_k}{n}\mu
\]
and set
\[
\textstyle
\gamma_{k+1}=\left(1-\frac{\alpha_k}{n}\right) \gamma_k+\frac{\alpha_k}{n}\mu.
\]
\item Compute $y^k$ as
\[
y^k ~=~ \textstyle \frac{1}{\frac{\alpha_k}{n}\gamma_k+\gamma_{k+1}} \left(
\frac{\alpha_k}{n}\gamma_k v^k + \gamma_{k+1} x^k \right).
\]
\item
Choose $i_k\in\{1,\ldots,n\}$ uniformly at random, and update
\[
x^{k+1} = y^{k} -\textstyle  \frac{1}{L_{i_k}} U_{i_k}\nabla_{i_k}f(y^{k}).
\vspace{-2ex}
\]
\item Set
\[
v^{k+1} = \textstyle \frac{1}{\gamma_{k+1}} \left(
\left(1-\frac{\alpha_k}{n}\right)\gamma_k v^k + \frac{\alpha_k}{n}\mu y^k
-\frac{\alpha_k}{L_{i_k}} U_{i_k} \nabla_{i_k} f(y^k) \right) .
\]
\end{enumerate}
\end{boxedminipage}
\end{center}

\gap

\begin{remark}
For the above algorithm, claim that $\gamma_k >0$ and $\alpha_k$ is well-defined for all $k$.
Indeed, let $\gamma >0$ be arbitrarily given and define
\[
h(\alpha) := \alpha^2 - \left(1-\frac{\alpha}{n}\right) \gamma - \frac{\alpha}{n} \mu, \qquad \forall \alpha
\ge 0.
\]
We observe that
\[
h(0)=-\gamma \ < \ 0, \quad\quad h(n) = n^2-\mu \ \ge \ 0,
\]
where the last inequality is due to $\mu \le 1$. Therefore, by continuity of $h$,
there exists some $\alpha^* \in (0,n]$ such that $h(\alpha^*)=0$. Moreover, if $\mu=0$,
we have $0<\alpha^*<n$. Using these observations and the definitions of $\alpha_k$ and $\gamma_k$,
it is not hard to see by induction that $\gamma_k>0$ and $\alpha_k$ is well-defined for all $k$.
\end{remark}

The above description of the ARCD method comes directly from the derivation
using \emph{randomized estimate sequence} we develop in Section~\ref{sec:res},
and is very convenient for the purpose of our convergence analysis.
For implementation in practice, one can simplify the notations and use an
equivalent algorithm described below.
In the simplified description, it is also clear that the ARCD method is
equivalent to the method~(5.1) in \cite[Section~5]{Nesterov12rcdm},
with the following correspondences between the symbols used.
\begin{center}
\begin{tabular}{l|ccccc}
\hline
This paper
& $\alpha_k$ & $\alpha_{k-1}$ & $\theta_k$ & $\beta_k$ & $\mu$ \\
\hline
\cite[(5.1)]{Nesterov12rcdm}
& $1/\gamma_k$ & $b_k/a_k$ & $\alpha_k$ & $\beta_k$ & $\sigma$ \\
\hline
\end{tabular}
\end{center}

\begin{center}
\begin{boxedminipage}{0.8\textwidth}
\vspace{1ex}
\centerline{\textbf{Algorithm:} ARCD$(x^0)$}
\vspace{1ex}
Set $v^0=x^0$, choose $\alpha_{-1}\in(0,n]$,
and repeat for $k=0,1,2,\ldots$
\vspace{-1ex}
\begin{enumerate} \itemsep 0pt
\item Compute $\alpha_k\in(0, n]$ from the equation
\[
\textstyle
\alpha_k^2 = \left(1-\frac{\alpha_k}{n}\right) \alpha^2_{k-1}
+\frac{\alpha_k}{n}\mu ,
\]
and set
\[
\textstyle
\theta_k = \frac{n\alpha_k-\mu}{n^2-\mu}, \qquad
\beta_k = 1-\frac{\mu}{n\alpha_k}  .
\]
\item Compute $y^k$ as
\[
y^k ~=~ \theta_k v^k + (1-\theta_k) x^k .
\]
\item
Choose $i_k\in\{1,\ldots,n\}$ uniformly at random, and update
\[
x^{k+1} = y^{k} -\textstyle \frac{1}{L_{i_k}} U_{i_k} \nabla_{i_k}f(y^{k}).
\vspace{-2ex}
\]
\item Set
\[
\textstyle
v^{k+1} = \beta_k v^k + (1-\beta_k) y^k
-\frac{1}{\alpha_kL_{i_k}} U_{i_k} \nabla_{i_k} f(y^k) .
\]
\end{enumerate}
\end{boxedminipage}
\end{center}

\gap 

At each iteration~$k$, the ARCD method generates $y^k$, $x^{k+1}$ and
$v^{k+1}$. One can observe that $x^{k+1}$ and $v^{k+1}$ depend on the realization of the random variable
\[
\xi_k = \{i_0, i_1,\ldots, i_k\}
\]
while $y^k$ depends on the realization of $\xi_{k-1}$.

We now state  a sharper expected-value type of convergence rate for the ARCD method
than the one  given in \cite{Nesterov12rcdm}. Its proof relies on a new technique called
{\it randomized estimate sequence} that will be developed in Subsection \ref{sec:res}. Therefore,
we postpone the proof to Subsection \ref{proof-thm}.

\begin{theorem}\label{thm:arcd-rate}
Let $f^\star$ be the optimal value of problem~(\ref{eqn:unconstr-min}), $R_0$
be defined in \eqref{R0}, and $\{x^k\}$ be the sequence generated by the ARCD method.
Then, for any $k\geq 0$, there holds:
\[
\bE_{\xi_{k-1}} [f(x^k)] - f^\star ~\leq~ \lambda_k
\left(f(x^0)-f^\star+\frac{\gamma_0R^2_0}{2}\right),
\]
where $\lambda_0=1$ and
$\lambda_k=\prod_{i=0}^{k-1} \left(1-\frac{\alpha_i}{n}\right)$.
In particular, if $\gamma_0\geq\mu$, then
\[
\lambda_k ~\leq~ \min \left\{ \left(1-\frac{\sqrt{\mu}}{n}\right)^k,
~\left(\frac{n}{n+k\frac{\sqrt{\gamma_0}}{2}}\right)^2 \right\}.
\]
\end{theorem}


\begin{remark}
We note that for $n=1$, the ARCD method reduces to a deterministic
accelerated full gradient method described in \cite[(2.2.8)]{Nesterov04book};
Our iteration complexity result above also becomes the same as the one given there.
\end{remark}

Nesterov \cite[Theorem~6]{Nesterov12rcdm} established the following convergence rate
for the above ARCD method:
\[
\bE_{\xi_{k-1}} [f(x^k)] - f^\star \ \leq \ \left\{\ba{lcl}
 \overbrace{\mu \left[2R^2_0+\frac{1}{n^2}(f(x^0)-f^\star)\right] \cdot
\left[\left(1+\frac{\sqrt{\mu}}{2n}\right)^{k+1}-\left(1-\frac{\sqrt{\mu}}{2n}\right)^{k+1}\right]^{-2}}^{a_\mu}
& \ \mbox{if} \ \mu >0, \\ [25pt]
\underbrace{\left(\frac{n}{k+1}\right)^2 \cdot \left[2R^2_0+\frac{1}{n^2}(f(x^0)-f^\star)\right]}_{a_0} & \ \mbox{otherwise}.
\ea \right.
\]
In view of Theorem~\ref{thm:arcd-rate}, our convergence rate is given by
\[
\bE_{\xi_{k-1}} [f(x^k)] - f^\star ~\leq~ \underbrace{\min \left\{\left(1-\frac{\sqrt{\mu}}{n}\right)^k,
~\left(\frac{n}{n+k\frac{\sqrt{\gamma_0}}{2}}\right)^2 \right\}
\left(f(x^0)-f^\star+\frac{\gamma_0R^2_0}{2}\right)}_{b_\mu}
\]
We now compare the above two rates by considering two cases: $\mu>0$ and $\mu=0$.

\bi
\item 
Case (1): $\mu>0$. We can observe that for sufficiently large $k$,
\[
a_\mu \ = \ O\left(\left(1+\frac{\sqrt{\mu}}{2n}\right)^{-2k}\right), \quad\quad
b_\mu \ = \ O\left(\left(1-\frac{\sqrt{\mu}}{n}\right)^k\right).
\]
It is easy to verify that
\[
\left(1+\frac{\sqrt{\mu}}{2n}\right)^{-2} \ > \ 1-\frac{\sqrt{\mu}}{n}
\]
and hence $a_\mu \gg  b_\mu$ when $k$ is sufficiently large, which implies that our rate is much tighter.
\item 
Case (2): $\mu=0$. For sufficiently $k$, we have
\[
\ba{lcl}
a_0 & \approx & (2n^2R^2_0+f(x^0)-f^\star)/k^2, \\ [8pt]
b_0 &\approx & \left(2n^2R^2_0 + \frac{4n^2}{\gamma_0}(f(x^0)-f^\star)\right)/k^2.
\ea
\]
Therefore, when $\gamma_0 > 4n^2$, we obtain $b_0  < a_0 $ for sufficiently large $k$,
which again implies that our rate is sharper.
\ei

\subsection{Randomized estimate sequence}
\label{sec:res}

In \cite{Nesterov04book}, Nesterov introduced a powerful
framework of \emph{estimate sequence} for the development and analysis
of accelerated full gradient methods. Here we extend it to a randomized 
block-coordinate descent setup, 
and use it to analyze the convergence rate of the ARCD method subsequently.

\begin{definition}\label{def:res}
Let $\phi_0(x)$ be a deterministic function and $\phi_k(x)$ be a random function
depending on $\xi_{k-1}$ for all $k \ge 1$, and $\lambda_k\geq0$ for all $k\geq 0$.
The sequence $\{(\phi_k(x),\lambda_k)\}_{k=0}^\infty$ is called a randomized
estimate sequence of function~$f(x)$ if
\beq \label{lim-lambda}
\lambda_k \to 0
\eeq
and for any $x\in\Re^N$ and all $k\geq 0$ we have
\beq \label{estimate-seq}
\bE_{\xi_{k-1}} [\phi_k(x)] ~\leq~ (1-\lambda_k) f(x) + \lambda_k \phi_0(x),
\eeq
where $\bE_{\xi_{-1}} [\phi_0(x)] \eqdef \phi_0(x)$.
\end{definition}

Here we assume $\{\lambda_k\}_{k\geq 0}$ is a deterministic sequence
that is independent of $\xi_k$.

\begin{lemma}\label{lem:res}
Let $x^\star$ be an optimal solution to~(\ref{eqn:unconstr-min})
and $f^\star$ be the optimal value. Suppose that $\{(\phi_k(x),\lambda_k)\}_{k=0}^\infty$ is
a randomized estimate sequence of function~$f(x)$. Assume that $\{x^k\}$ is a sequence such that
for each $k\geq0$,
\beq \label{rand-upbdd}
\bE_{\xi_{k-1}} [f(x^k)] ~\leq~ \min_x \bE_{\xi_{k-1}} [\phi_k(x)],
\eeq
where $\bE_{\xi_{-1}} [f(x^0)] \eqdef f(x^0)$. Then we have
\[
\bE_{\xi_{k-1}} [f(x^k)] - f^\star ~\leq~
\lambda_k \left( \phi_0(x^\star) - f^\star\right) ~\to~ 0.
\]
\end{lemma}

\begin{proof}
Since $\{(\phi_k(x),\lambda_k)\}_{k=0}^\infty$ is a randomized estimate sequence of $f(x)$, it follows from
\eqref{estimate-seq} and \eqref{rand-upbdd} that
\begin{eqnarray*}
\bE_{\xi_{k-1}} [f(x^k)] &\leq& \min_x \bE_{\xi_{k-1}} [\phi_k(x)] \\
&\leq& \min_x \left\{ (1-\lambda_k) f(x) + \lambda_k \phi_0(x) \right\} \\
&\leq& (1-\lambda_k) f(x^\star) + \lambda_k \phi_0(x^\star) \\
&=&  f^\star + \lambda_k ( \phi_0(x^\star) - f^\star ),
\end{eqnarray*}
which together with \eqref{lim-lambda} implies that the conclusion holds.
\end{proof}

\gap

As we will see next, our construction of the randomized estimate sequence
satisfies a stronger condition, i.e.,
\[
\bE_{\xi_{k-1}} f(x^k) ~\leq~ \bE_{\xi_{k-1}} [\min_x\phi_k(x)].
\]
This implies that the assumption in Lemma~\ref{lem:res}, namely, \eqref{rand-upbdd} holds due to
\[
\bE_{\xi_{k-1}} [\min_x\phi_k(x)] ~\leq~ \min_x \bE_{\xi_{k-1}} [\phi_k(x)].
\]

\begin{lemma}\label{lem:construct-res}
Assume that $f$ satisfies Assumption~\ref{asmp:coord-smooth} with convexity parameter $\mu \ge 0$.
In addition, suppose that
\begin{itemize}
\item $\phi_0(x)$ is an arbitrary deterministic function on $\Re^N$;
\item $\{y^k\}_{k=1}^\infty$ is a sequence in $\Re^N$ such that
$y^k$ depends on $\xi_{k-1}$;
\item $\{\alpha_k\}_{k=1}^\infty$ is independent of $\xi_k$ and satisfies
$\alpha_k\in(0,n)$ for all $k\geq0$ and $\sum_{k=0}^\infty\alpha_k=\infty$.
\end{itemize}
Then the pair of sequences $\{\phi_k(x)\}_{k=0}^\infty$ and
$\{\lambda_k\}_{k=0}^\infty$ constructed by setting $\lambda_0=1$ and
\begin{eqnarray}
\lambda_{k+1} &=& \left(1-\frac{\alpha_k}{n}\right) \lambda_k, \label{lambda-def}\\
\phi_{k+1}(x) &=& \left(1-\frac{\alpha_k}{n}\right) \phi_k(x) + \alpha_k\left(
\frac{1}{n}f(y^k) + \langle \nabla_{i_k}f(y^k), x_{i_k}-y^k_{i_k} \rangle
+\frac{\mu}{2 n}\|x-y^k\|_L^2 \right) \label{phi-def},
\end{eqnarray}
is a randomized estimate sequence of~$f(x)$.
\end{lemma}

\begin{proof}
It follows from \eqref{lambda-def} and $\lambda_0=1$ that
$\lambda_k = \prod_{i=0}^{k-1} (1-\alpha_i/n)$ for $k \ge 1$. Then we have
\[
\log\lambda_k ~=~ \sum_{i=0}^{k-1} \log\left(1-\frac{\alpha_i}{n}\right) ~\le~ -\frac1n \sum_{i=0}^{k-1} \alpha_i
~\to~ -\infty
\]
due to $\sum_{i=0}^\infty\alpha_i=\infty$. Hence, $\lambda_k \to 0$.
We next prove by induction that \eqref{estimate-seq} holds for all $k\ge 0$. Indeed, for $k=0$, we
know that $\lambda_0=1$ and hence
\[
\bE_{\xi_{-1}}[\phi_0(x)] ~=~ \phi_0(x)  ~=~ (1-\lambda_0) f(x) + \lambda_0 \phi_0(x),
\]
that is, \eqref{estimate-seq} holds for $k=0$. Now suppose it holds for some $k\geq 0$. Using \eqref{phi-def}, we
obtain that
\begin{eqnarray*}
\bE_{\xi_k} [\phi_{k+1}(x)]
&=& \bE_{\xi_{k-1}}\left[\bE_{i_k} [\phi_{k+1}(x)] \right] \\
&=& \bE_{\xi_{k-1}}\bigg[ \left(1-\frac{\alpha_k}{n}\right) \phi_k(x)
   + \alpha_k \bigg( \frac{1}{n}f(y^k) + \bE_{i_k}\left[
   \langle \nabla_{i_k}f(y^k), x_{i_k}-y^k_{i_k} \rangle \right] \\
&&  \hspace{6cm} +\frac{\mu}{2 n}\|x-y^k\|_L^2 \bigg) \bigg] \\
&=& \bE_{\xi_{k-1}}\left[ \left(1-\frac{\alpha_k}{n}\right) \phi_k(x)
   + \frac{\alpha_k}{n} \left( f(y^k) + \langle \nabla f(y^k), x-y^k \rangle
   +\frac{\mu}{2}\|x-y^k\|_L^2 \right) \right] \\
&\leq& \bE_{\xi_{k-1}}\left[ \left(1-\frac{\alpha_k}{n}\right) \phi_k(x)
   + \frac{\alpha_k}{n} f(x) \right],
\end{eqnarray*}
where the last inequality is due to convexity of~$f$.
Using the induction hypothesis, we have
\begin{eqnarray*}
\bE_{\xi_k} [\phi_{k+1}(x)]
&\leq& \left(1-\frac{\alpha_k}{n}\right) \bigl(
  (1-\lambda_k)f(x)+\lambda_k\phi_0(x) \bigr) + \frac{\alpha_k}{n} f(x) \\
&=& \left(1-\left(1-\frac{\alpha_k}{n}\right)\lambda_k\right) f(x)
  + \left(1-\frac{\alpha_k}{n}\right) \lambda_k \phi_0(x) \\
&=& (1-\lambda_{k+1}) f(x)  + \lambda_{k+1} \phi_0(x)
\end{eqnarray*}
and hence \eqref{estimate-seq} also holds for $k+1$.  This completes the proof.
\end{proof}

\gap

\begin{lemma}\label{lem:quadratic-res}
Let $\phi_0(x) = \phi_0^\star + \frac{\gamma_0}{2} \|x-v^0\|_L^2$.
Then the randomized estimate sequence constructed in
Lemma~\ref{lem:construct-res} preserves the canonical form of the functions,
i.e., for all $k\geq 0$,
\begin{equation}\label{eqn:res-canonical}
\phi_k(x) = \phi_k^\star + \frac{\gamma_k}{2} \|x-v^k\|_L^2,
\end{equation}
where the sequences $\{\gamma_k\}$, $\{v^k\}$ and $\{\phi_k^\star\}$ are
defined as follows:
\begin{eqnarray}
\gamma_{k+1} &=& \left(1-\frac{\alpha_k}{n}\right) \gamma_k
    + \frac{\alpha_k}{n} \mu, \label{gammak}\\
v^{k+1} &=& \frac{1}{\gamma_{k+1}} \left( \left(1-\frac{\alpha_k}{n}\right)
    \gamma_k v^k + \frac{\alpha_k}{n} \mu y^k - \frac{\alpha_k}{L_{i_k}}
    U_{i_k} \nabla_{i_k} f(y^k) \right) \label{vk}\\
\phi_{k+1}^\star &=& \left(1-\frac{\alpha_k}{n}\right) \phi_k^\star
    +\frac{\alpha_k}{n} f(y^k) - \frac{\alpha_k^2}{2\gamma_{k+1}L_{i_k}}
    \|\nabla_{i_k}f(y^k)\|^2 \nonumber \\
&& +\frac{\alpha_k\left(1-\frac{\alpha_k}{n}\right)\gamma_k}{\gamma_{k+1}}
    \left( \frac{\mu}{2n}\|y^k-v^k\|_L^2 + \langle \nabla_{i_k}f(y^k),
    v^k_{i_k} - y^k_{i_k} \rangle \right) \label{phik-star}
\end{eqnarray}
\end{lemma}

\begin{proof}
First we observe that $\phi_k(x)$ is a convex quadratic function due to \eqref{phi-def}
and the definition of $\phi_0(x)$.  We now prove by induction that for $\phi_k$ is given by
\eqref{eqn:res-canonical} all $k\geq 0$. Clearly, \eqref{eqn:res-canonical} holds for $k=0$.
Suppose now that it holds for some $k \ge 0$. It follows that the Hessian of $\phi_k(x)$ is a
block-diagonal matrix given by
\[
\nabla^2 \phi_k(x)~=~\gamma_k\,\diag\left(L_1 I_{N_1},\ldots,L_n I_{N_n}\right).
\]
Using this relation, \eqref{phi-def} and \eqref{gammak}, we have
\begin{eqnarray}
\nabla^2 \phi_{k+1}(x)
&=& \left(1-\frac{\alpha_k}{n}\right) \nabla^2\phi_k(x) +
    \frac{\alpha_k}{n}\mu\, \diag\left(L_1 I_{N_1},\ldots,L_n I_{N_n}\right) \nonumber  \\
&=& \gamma_{k+1}\,\diag\left(L_1 I_{N_1},\ldots,L_n I_{N_n}\right) \label{Hessian-phi}.
\end{eqnarray}
Using the induction hypothesis by substituting \eqref{eqn:res-canonical} into \eqref{phi-def},
we can write $\phi_{k+1}(x)$ as
\begin{eqnarray}
\phi_{k+1}(x)
&=& \left(1-\frac{\alpha_k}{n}\right) \left( \phi_k^\star
    + \frac{\gamma_k}{2} \|x-v^k\|_L^2 \right) \nonumber \\
&& + \alpha_k\left( \frac{1}{n}f(y^k) + \langle \nabla_{i_k}f(y^k),
    x_{i_k}-y^k_{i_k} \rangle +\frac{\mu}{2 n}\|x-y^k\|_L^2 \right), \label{phik-form}
\end{eqnarray}
which together with \eqref{vk} implies
\beq \label{deriv}
\nabla \phi_{k+1}(v^{k+1}) = \left(1-\frac{\alpha_k}{n}\right) \gamma_k \sum_{i=1}^n U_i L_i(v^{k+1}_i-v^k_i)
+ \alpha_k U_{i_k} \nabla_{i_k} f(y^k) + \frac{\alpha_k}{n}\mu
\sum_{i=1}^n U_i L_i(v^{k+1}_i-y^k_i) = 0.
\eeq
Letting $x=y^k$ in \eqref{phik-form}, one has
\[
\phi_{k+1}(y^k) = \left(1-\frac{\alpha_k}{n}\right) \left( \phi_k^\star
+ \frac{\gamma_k}{2}\|y^k-v^k\|_L^2 \right) + \frac{\alpha_k}{n} f(y^k).
\]
In view of \eqref{vk}, we have
\[
v^{k+1} - y^k
~=~ \frac{1}{\gamma_{k+1}} \left( \left(1-\frac{\alpha_k}{n}\right) \gamma_k
(v^k-y^k) - \frac{\alpha_k}{L_{i_k}} U_{i_k}\nabla_{i_k} f(y^k) \right),
\]
and hence
\begin{eqnarray*}
\frac{\gamma_{k+1}}{2}\|y^k-v^{k+1}\|_L^2
&=& \frac{1}{2\gamma_{k+1}}\biggl(\left(1-\frac{\alpha_k}{n}\right)^2\gamma_k^2
    \|y^k-v^k\|_L^2 + \frac{\alpha_k^2}{L_{i_k}}\|\nabla_{i_k} f(y^k)\|^2 \\
&& \qquad\qquad- 2 \alpha_k \left(1-\frac{\alpha_k}{n}\right) \gamma_k
    \left\langle \nabla_{i_k}f(y^k), v^k_{i_k} - y^k_{i_k} \right\rangle \bigg).
\end{eqnarray*}
In addition, using \eqref{gammak} we obtain that
\[
\left(1-\frac{\alpha_k}{n}\right) \frac{\gamma_k}{2}
-\frac{1}{2\gamma_{k+1}} \left(1-\frac{\alpha_k}{n}\right)^2\gamma_k^2
~=~ \frac{1}{2\gamma_{k+1}} \left(1-\frac{\alpha_k}{n}\right) \gamma_k
\frac{\alpha_k}{n} \mu.
\]
By virtue of the above relations and \eqref{phik-star},  it is not hard to
conclude that
\[
\phi_{k+1}(y^k) =  \phi_{k+1}^\star + \frac{\gamma_{k+1}}{2}\|y^k-v^{k+1}\|_L^2,
\]
which, together with \eqref{Hessian-phi}, \eqref{deriv} and the fact that $\phi_{k+1}$
is quadratic, implies that
\[
\phi_{k+1}(x) = \phi_{k+1}^\star + \frac{\gamma_{k+1}}{2}\|x-v^{k+1}\|_L^2.
\]
Therefore, the conclusion holds.
\end{proof}

\subsection{Proof of Theorem~\ref{thm:arcd-rate}}
\label{proof-thm}

Let $\phi_0(x)=f(v^0)+\gamma_0 \|x-v^0\|^2_L/2$, $\{y^k\}$ and $\{\alpha_k\}$ be generated
in the ARCD method. In addition, let $\{(\phi_k(x),\lambda_k\}$ be the randomized estimate
sequence of $f(x)$ generated as in Lemma~\ref{lem:construct-res} by using such
$\{y^k\}$ and $\{\alpha_k\}$.

First we prove by induction that for all $k\geq 0$,
\begin{equation}\label{eqn:res-hypothesis}
\bE_{\xi_{k-1}} [f(x^k)]
~\leq~ \bE_{\xi_{k-1}} \left[\left\{ \phi_k^\star = \min_x \phi_k(x) \right\}\right].
\end{equation}
For $k=0$, using $v^0=x^0$, the definition of $\phi_0(x)$ and
$\bE_{\xi_{-1}} [f(x^0)]=f(x^0)$, we have
\[
\bE_{\xi_{-1}} [f(x^0)] ~=~ f(x^0) ~=~ f(v^0) ~=~ \phi_0^\star,
\]
and hence \eqref{eqn:res-hypothesis} holds for $k=0$. Now suppose it holds for some
$k\geq 0$. It follows from \eqref{phik-star} that
\begin{eqnarray}
\bE_{\xi_k}[\phi_{k+1}^\star]
&=& \bE_{\xi_{k-1}}\left[ \bE_{i_k} [\phi_{k+1}^\star] \right] \nonumber\\
&=& \bE_{\xi_{k-1}}\Biggl[ \left(1-\frac{\alpha_k}{n}\right) \phi_k^\star
    +\frac{\alpha_k}{n} f(y^k) - \frac{\alpha_k^2}{2\gamma_{k+1}}
    \bE_{i_k}\left[ \frac{1}{L_{i_k}}\|\nabla_{i_k}f(y^k)\|^2 \right]
\label{eqn:expect-min-recursion}\\
&&\qquad\qquad
    +\frac{\alpha_k\left(1-\frac{\alpha_k}{n}\right)\gamma_k}{\gamma_{k+1}}
    \left( \frac{\mu}{2n}\|y^k-v^k\|_L^2
    + \bE_{i_k}\left[ \langle \nabla_{i_k}f(y^k),
    v^k_{i_k} - y^k_{i_k} \rangle \right] \right) \Biggr] . \nonumber
\end{eqnarray}
Let\
\[
d_i(y^k) ~=~ -\frac{1}{L_i} \nabla_i f(y^k), \qquad i=1,\ldots,n,
\]
and $d(y^k) =\sum_{i=1}^n U_i d_i(y^k)$. Then we have
\[
\bE_{i_k}\left[ \frac{1}{L_{i_k}}\|\nabla_{i_k}f(y^k)\|^2 \right]
~=~ \frac{1}{n}\|d(y^k)\|_L^2.
\]
Moreover,
\[
\bE_{i_k}\left[ \langle \nabla_{i_k}f(y^k), v^k_{i_k}-y^k_{i_k}\rangle\right]
~=~ \frac{1}{n} \langle \nabla f(y^k), v^k-y^k \rangle .
\]
Using these two equalities and dropping the term $\|y^k-v^k\|_L^2$
in~(\ref{eqn:expect-min-recursion}), we arrive at
\begin{eqnarray*}
\bE_{\xi_k}[\phi_{k+1}^\star]
&\geq& \bE_{\xi_{k-1}}\Bigg[ \left(1-\frac{\alpha_k}{n}\right) \phi_k^\star
    +\frac{\alpha_k}{n} f(y^k)-\frac{\alpha_k^2}{2n\gamma_{k+1}}\|d(y^k)\|_L^2\\
&&  \quad\qquad
    +\frac{\alpha_k}{n}\left(1-\frac{\alpha_k}{n}\right)\frac{\gamma_k}
    {\gamma_{k+1}} \langle \nabla f(y^k), v^k-y^k \rangle \Bigg].
\end{eqnarray*}
By the induction hypothesis and the convexity of $f$, we obtain that
\[
\bE_{\xi_{k-1}} [\phi_k^\star] ~\geq~ \bE_{\xi_{k-1}} [f(x^k)]
~\geq~ \bE_{\xi_{k-1}}[f(y^k) +\langle \nabla f(y^k),x^k-y^k\rangle].
\]
Combining the above two inequalities gives
\begin{eqnarray*}
\bE_{\xi_k}[\phi_{k+1}^\star]
&\geq& \bE_{\xi_{k-1}}\Bigg[
    f(y^k)-\frac{\alpha_k^2}{2n\gamma_{k+1}}\|d(y^k)\|_L^2\\
&&  \quad\qquad
    +\left(1-\frac{\alpha_k}{n}\right) \left\langle \nabla f(y^k),
    \frac{\alpha_k\gamma_k}{n\gamma_{k+1}} (v^k-y^k) + (x^k-y^k)
    \right\rangle \Bigg].
\end{eqnarray*}
Recall that
\[
y^k ~=~ \frac{1}{\frac{\alpha_k}{n}\gamma_k+\gamma_{k+1}} \left(
\frac{\alpha_k}{n}\gamma_k v^k + \gamma_{k+1} x^k \right).
\]
This relation together with the above inequality yields
\[
\bE_{\xi_k}[\phi_{k+1}^\star]
~\geq~ \bE_{\xi_{k-1}}\left[
    f(y^k)-\frac{\alpha_k^2}{2n\gamma_{k+1}}\|d(y^k)\|_L^2 \right].
\]
Also, we observe that $\alpha_k^2=\gamma_{k+1}$. Substituting it into the
above inequality gives
\[
\bE_{\xi_k}[\phi_{k+1}^\star]
~\geq~ \bE_{\xi_{k-1}}\left[ f(y^k)-\frac{1}{2n}\|d(y^k)\|_L^2 \right].
\]
In addition, notice that
\[
x^{k+1} ~=~ y^k - \frac{1}{L_{i_k}} U_{i_k}\nabla_{i_k} f(y^k)
~=~ y^k + U_{i_k} d_{i_k}(y^k),
\]
which together with Corollary~\ref{cor:monotone} yields
\[
\bE_{\xi_k}[\phi_{k+1}^\star] ~\geq~ \bE_{\xi_{k-1}}\left[
\bE_{i_k} f(x^{k+1}) \right] ~=~ \bE_{\xi_k} [f(x^{k+1})].
\]
Therefore, (\ref{eqn:res-hypothesis}) holds for all $k+1$.
Further, by Lemma~\ref{lem:res}, we have
\[
\bE_{\xi_{k-1}} [f(x^k)] - f^\star ~\leq~ \lambda_k
\left(f(x^0)-f^\star+\frac{\gamma_0}{2}\|x^0-x^\star\|_L^2\right).
\]

Finally, we estimate the decay of $\lambda_k$, using the same arguments 
in the proof of \cite[Lemma~2.2.4]{Nesterov04book}.
Here we assume $\gamma_0\geq\mu$
(it suffices to set $\gamma_0=1$ because $\mu\leq 1$).
Indeed, if $\gamma_k\geq \mu$, then
\[
\gamma_{k+1}
~=~ \left(1-\frac{\alpha_k}{n}\right) \gamma_k + \frac{\alpha_k}{n} \mu
~\geq~ \mu.
\]
So we have $\gamma_k\geq\mu$ for all $k\geq 0$.
Since $\alpha_k^2=\gamma_{k+1}$,
we have $\alpha_k\geq\sqrt{\mu}$ for all $k\geq 0$.
Therefore,
\[
\lambda_k ~=~ \prod_{i=0}^{k-1} (1-\frac{\alpha_i}{n})
~\leq~ \left(1-\frac{\sqrt{\mu}}{n}\right)^k .
\]
In addition, we have $\gamma_k\geq\gamma_0\lambda_k$.
To see this, we note $\gamma_0=\gamma_0\lambda_0$ and use induction
\[
\gamma_{k+1} ~\geq~ \left(1-\frac{\alpha_k}{n}\right) \gamma_k
~\geq~\left(1-\frac{\alpha_k}{n}\right) \gamma_0\lambda_k
~=~ \gamma_0\lambda_{k+1}.
\]
This implies
\begin{equation}\label{eqn:alpha-lb}
\alpha_k ~=~\sqrt{\gamma_{k+1}} ~\geq~ \sqrt{\gamma_0\lambda_{k+1}}.
\end{equation}
Since $\{\lambda_k\}$ is a decreasing sequence, we have
\begin{eqnarray*}
\frac{1}{\sqrt{\lambda_{k+1}}} - \frac{1}{\sqrt{\lambda_k}}
&=& \frac{\sqrt{\lambda_k}-\sqrt{\lambda_{k+1}}}
    {\sqrt{\lambda_k}\sqrt{\lambda_{k+1}}}
~=~\frac{\lambda_k-\lambda_{k+1}}{\sqrt{\lambda_k}\sqrt{\lambda_{k+1}}
    (\sqrt{\lambda_k}+\sqrt{\lambda_{k+1}})} \\
&\geq& \frac{\lambda_k-\lambda_{k+1}}{2\lambda_k\sqrt{\lambda_{k+1}}}
~=~ \frac{\lambda_k-\left(1-\frac{\alpha_k}{n}\right)\lambda_k}
    {2\lambda_k\sqrt{\lambda_{k+1}}}
~=~ \frac{\frac{\alpha_k}{n}}{2\sqrt{\lambda_{k+1}}} .
\end{eqnarray*}
Combining with~(\ref{eqn:alpha-lb}) gives
\[
\frac{1}{\sqrt{\lambda_{k+1}}} - \frac{1}{\sqrt{\lambda_k}}
~\geq~ \frac{\sqrt{\gamma_0}}{2n}.
\]
By further noting $\lambda_0=1$, we obtain
\[
\frac{1}{\sqrt{\lambda_k}} ~\geq~ 1 + \frac{k}{n}\frac{\sqrt{\gamma_0}}{2}.
\]
Therefore
\[
\lambda_k ~\leq~ \left(\frac{n}{n+k\frac{\sqrt{\gamma_0}}{2}}\right)^2 .
\]
This completes the proof for Theorem~\ref{thm:arcd-rate}.

%
%
%

\end{document}